\documentclass[10pt,reqno]{amsart}
\usepackage{enumerate}

     \makeatletter
     \def\section{\@startsection{section}{1}%
     \z@{.7\linespacing\@plus\linespacing}{.5\linespacing}%
     {\bfseries
     \centering
     }}
     \def\@secnumfont{\bfseries}
     \makeatother
\newtheorem{theorem}{Theorem}[section]
\newtheorem{lemma}[theorem]{Lemma}
\newtheorem{proposition}[theorem]{Proposition}

\theoremstyle{definition}
\newtheorem{definition}[theorem]{Definition}

\newtheorem{hypothesis}{Hypothesis}
\theoremstyle{remark}
\newtheorem{remark}[theorem]{Remark}
\numberwithin{equation}{section} 

\newcommand{\Real}{\mathbb{R}}

\newcommand{\R}{\mathbb{R}}

\begin{document}

\title[Strong Kac's chaos in the mean-field BEC]{ Strong Kac's chaos in the mean-field Bose-Einstein Condensation}

\author[S. Albeverio]{\bf Sergio Albeverio}

\author[F.C. De Vecchi]{\bf Francesco C. De Vecchi}
\address{Institute for Applied Mathematics and Hausdorff Center for Mathematics, Rheinische Friedrich-Wilhelms-Universit\"at Bonn.
Endenicher Allee 60, 53115 Bonn, Germany.}
\email{albeverio@iam.uni-bonn.de, francesco.deveccchi@uni-bonn.de}

\author[A. Romano]{\bf Andrea Romano}

\author[S. Ugolini]{\bf Stefania Ugolini}
\address{Dipartimento di Matematica, Universit\`a di Milano. Via
Saldini 50, Milano, Italy.}
\email{andrea.romano4@studenti.unimi.it, stefania.ugolini@unimi.it}

\begin{abstract}
A stochastic approach to the (generic) mean-field limit in Bose-Einstein Condensation is described and the convergence of the ground state energy as well as of its components are established.
 For the one-particle process on the path space
a total variation convergence result is proved. A strong form of Kac's chaos on path-space for the $k$-particles probability measures are derived from the previous energy convergence by purely probabilistic techniques notably using a simple chain-rule of the relative entropy. The Fisher's information chaos of the fixed-time marginal probability density under the generic mean-field scaling limit and the related entropy chaos result are also deduced.
\end{abstract}



\maketitle

\section{Introduction}

We consider the problem of justifying the general mean-field approach to Bose-Einstein Condensation in the ground state framework, starting from the $N$ body Hamiltonian for $N$ Bose particles  and by  performing a suitable limit of infinitely many particles. We say general because we include the cases corresponding to $0 \leq \beta < 1$, where $\beta$ is the parameter that allows to model the $N$-dependence of the range of the interacting potential. The case $\beta=0$ corresponds properly to the mean-field approximation where the potential range is fixed and the intensity of the interacting potential decreases as $1/N$. Differently the regime corresponding to $0 < \beta < 1$ is more difficult  since the interaction potential goes to a delta function in the sense of the  measures convergence. This case is not very well studied and it  is usually denoted as the non linear Schroedinger limit (see \cite{Lewin}).
There are many results and some quantitative estimates of the convergence rate for small values of $\beta$ (see \cite{Lewin,Rougerie} and references therein).
We face the general mean-field convergence problem by using the hard results for the case $\beta=1$, known as Gross-Pitaevskii scaling limit, obtained in \cite{Lieb,Lieb2}  and, recently,  in \cite{Nam}. We prove the convergence of the one-particle ground-state energy to the ground-state energy of the non-linear Schroedinger functional for the case of purely repulsive interacting potential (Theorem \ref{theorem1}). After this, we are also able to discriminate how the single terms of the energy converge (Theorem \ref{theorem2}).\\ 

Nelson's Stochastic Mechanics (\cite{Nelson1,Nelson2,CarlenN})  allows to rigorously associate  a system of $N$ interacting diffusions to the $N$ body Hamiltonian (\cite{Carlen}) and consequently to consider the convergence problem of the one-particle probability measure on the path-space to the limit measure of McKean-Vlasov type. In this paper we prove that the convergence holds in total variation (Theorem \ref{theorem3}), which is stronger with respect to the weak convergence recently obtained in \cite{ADU} for the Gross-Pitaevskii scaling limit. 
Successively we establish both the usual Kac's chaos and a strong form of  Kac's chaos for the law of the $N$ interacting Nelson diffusions (Theorem \ref{theorem5}). It is well-known indeed that Kac's chaos is usually expressed as the week convergence of the law of any $k$-components diffusions to the asymptotic $k$-product measure. In our paper we state that the cited convergence on path space is in total variation sense (Theorem \ref{theorem4} and Theorem \ref{theorem5}).
The proof, based on Girsanov Theorem and on  the relative entropy between the two involved probability measures,  is essentially a probabilistic one.
The relevant analytic result used in both the proofs is the $ L^2$-convergence of the difference of the drifts which can be deduced from Theorem \ref{theorem2} without the usual assumptions of bounded or Lipschitz drifts (that are not  satisfied in our quantum mathematical setting).\\

The plan of the present paper is the following. In Section \ref{section_quantum} we introduce  the quantum framework of the derivation of the non-linear Schroedinger model from the initial $N$ body Hamiltonian describing the $N$ Bose particles through a suitable scaling limit of general mean-field type and we prove the weak-$L^1$ convergence of the ground-state energies to the asymptotic non linear energy as well as the convergence of the single terms of the associated energy functional.

In Section \ref{section_nelson} we briefly introduce the Nelson-Carlen scheme of Stochastic Mechanics. The total variation convergence of the one-particle measure on path-space is proved in Section \ref{section_totalvariation}. In Section \ref{section_kac_chaos}  both the usual and the strong Kac's chaos is established and in the  Section \ref{section_fisher} we derive the entropy chaos from Fisher chaos  for the fixed time symmetric probability law on the product space $\Real^{3N}$.  

\section{Convergence of the mean-field quantum energy functional}\label{section_quantum}

\vskip4pt
We want to study a thermodynamic limit of a system of Bosons in $\mathbb{R}^3$. For this reason we consider the following $N$-body quantum Hamiltonian
\begin{equation} \label{HN}
H_N=\sum _{i=1}^{N}\left(-\frac {\hbar^2}{2m}{\triangle}_{i}+V({\bf
r_i})\right)+ \sum _{1\leq i< j \leq N} v_N(\bf r_i- \bf r_j)
\end{equation}
\noindent where $V$ is a confining potential,  $v_N$ a pair-wise
repulsive interaction potential (depending on the number of particles for obtaining a meaningful thermodynamic limit) and ${\bf r}_i \in \R^3,
i=1,...,N$. It operates on symmetric wave functions $\Psi$ in the
complex $L^2(\Real^{3N})$-space in order to satisfy the symmetry
permutation prescription for Bose particles.

We consider the mean quantum mechanical energy
\begin{equation}\label{quantenergy}
{\mathcal E}_N(\Psi)=<\Psi, H_N \Psi>, \quad \Psi \in H^1(\R^{3N})
\end{equation}
If  there exists a minimizing
function $\Psi_N$ of ${\mathcal E}_N$  it is called a \textit{ground state}
and the corresponding energy $ E_N[\Psi_N]$ given by
$$E_N[\Psi_N]:=\inf\left\{\mathcal { E}(\Psi): \int_{\Real^{3N}} |\Psi|^2d{\bf r_1 }...d{\bf r_2 }=1\right\}$$
is known as \textit{ground state energy}.

\noindent Under suitable assumptions on the potentials $V$ and $v_N$
one can prove the existence of the ground state $\Psi_N$ for
\eqref{HN}. Uniqueness of the ground state is to be understood as
uniqueness  apart from an \textit{overall phase}.

The Bose-Einstein condensate is obtained as the thermodynamic limit of the previous $N$-body problem and it  is usually
described by a wave $\phi \in H^1(\R^3)$, also
called wave function of the condensate, which is the minimizer of
the Hartree or of the non-linear Schroedinger (energy) functional (depending on the dependence of $v_N$ from $N$) 
\begin{multline} \label{GPfuncbetazero}
{\mathcal E}^{H}[\phi] = \int  \left(\frac {\hbar^2}{2m}|\nabla \phi({\bf r})|^2 + V(r)|\phi({\bf r})|^2\right)d{\bf r} + \\
 \int \int |\phi({\bf r})|^2v_0({\bf r}-{\bf r}_1)|\phi({\bf r}_1)|^2d{\bf r}d{\bf r}_1, 
 \end{multline}

\begin{equation} \label{GPfuncbeta}
{\mathcal E}^{nlS}[\phi] = \int  \left(\frac {\hbar^2}{2m}|\nabla \phi({\bf r})|^2 + V(r)|\phi({\bf r})|^2 +  g  |\phi({\bf r})|^4\right)d{\bf r}, 
 \end{equation}
 \noindent under the $L^2$-normalization condition
  $$ \int_{\R^3} |\phi({\bf r})|^2d{\bf r}=1$$
and where $$g=\int v_0(x)dx.$$ 

Here $v_0$ denotes the potential of the interaction between two particles before the thermodynamic limit is taken and its relation with $v_N$ is given by the following equation 
\begin{equation}\label{eq:vn} 
v_N({\bf r})=\frac{N^{3\beta}}{ N-1}v_0(N^{\beta}{\bf r}), \quad \quad 0 \le \beta  < 1
\end{equation}

From now on we consider the more difficult case given by $\mathcal{E}^{nlS}$ which implies $\beta>0$. 

We denote by $E_{nlS}$ the minimum of the energy \eqref{GPfuncbeta} and by  $\phi_{nlS}$ the minimizer which solves the stationary cubic non-linear equation (called non-linear Schroedinger
equation, nlS, or Gross-Pitaevskii equation, GP) (\cite{Lewin}, or \cite{Gross,Pitae})

 \begin{equation}\label{GP}
  -\frac {\hbar^2}{2m}\triangle \phi +V\phi + 2g |\phi|^2 \phi =\lambda
\phi
 \end{equation}
  \noindent $\lambda$, the real-valued Lagrange multiplier of the normalization constraint, is usually called  chemical potential. For the GP case one can
prove that $\phi$ is continuously differentiable and strictly
positive (\cite{Lieb}).

A stochastic quantization approach for the system
of $N$ interacting Bose particles has been faced for the first
time  in \cite{LM}. 

In \cite{MU} it has been proved that to the
$N$-body problem associated to $H_N$ there correspond  a well
defined diffusion process  describing the motion of the
single particle in the condensate, under the
Gross-Pitaevskii scaling limit as introduced in \cite{Lieb}, which
allows to prove the existence of an exact Bose-Einstein
condensation for the ground state of $H^N$
(see \cite{Lieb,Lieb2}). For the time-dependent derivation of
the Gross-Pitaevskii equation see \cite{Adami} and \cite{Erdos}. For the non linear Schroedinger case see \cite{Erdos1}.

\noindent For simplicity of notations, let us put $\hbar=2m=1$.

We consider the mean energy \eqref{quantenergy} 
$${\mathcal E}_N[\Psi_N]=\int \sum_{i=1}^N(|\nabla_i \Psi_N|^2+V({\bf r}_i))| \Psi_N|^2+
\sum_{1\leq i< j\leq N} v_N({\bf r}_i-{\bf r}_j)| \Psi_N|^2d{\bf r}_1\cdot\cdot\cdot{\bf
r}_N$$  \vskip5pt

\noindent We assume

h1) $V(|\bf{r}_i|)$ is locally bounded, continuous, strictly positive and going to
infinity when $|\bf{r}_i|$ goes to infinity.

h2) $v_0$ is smooth, compactly supported, non negative, spherically
symmetric.

 \vskip 5pt
\begin{remark}
The scaling case with $\beta=1$ does not belong to the mean-field regime. It is known as Gross-Pitaevskii scaling limit and it involves the scattering length of the interaction potential. The convergence of the ground state energy in this setting is difficult and has been provided by \cite{Lieb,Lieb2,Nam}.
\end{remark}
In this paper we propose a proof of the above convergence result  in the thermodynamic limit for a generic mean-field case (i.e. with $0\le \beta < 1$), but only under the assumption h2) corresponding to a positive-definite interaction, by taking advantage of the hard results of the GP regime.

\vskip 5mm

\begin{theorem}\label{theorem1} Under the
previous hypothesis h1), h2) we have that
\begin{equation} \label{LimE}
\lim_{N\uparrow \infty}\frac{E_N[\Psi_N]}{N}=E_{nlS}[\phi_{nlS}]
\end{equation}
and
\begin{equation} \label{Limrho}
\lim_{N\uparrow \infty}\int |\Psi_N|^2d{\bf r}_2\cdot\cdot\cdot{\bf
r}_N =|\phi_{nlS}|^2
\end{equation}
where  $ \phi_{nlS}$ is the minimizer
of the non-linear Schroedinger functional \eqref{GPfuncbeta} and the
convergence is in the weak $L^1(\Real^3)$ sense.
\end{theorem}

\begin{remark} The one-particle marginal density $\rho^{(1)}_N$
converges weakly to $\rho_{nlS}$ in the sense that the probability
measures $\rho^{(1)}_Nd{\bf r}$ weakly converge as $N\rightarrow
\infty$ towards the probability measure $\rho_{nlS}d{\bf r}$ on
$\R^3$.
\end{remark}

\begin{proof}
First of all we note that the following estimate 
$$E_{nlS}+CN^{-\beta}\ge \frac{E_N}{N}$$
is proved in (\cite{Lewin}, Proposition 2.3 or Theorem 2.4). This means that $\lim_{N \uparrow \infty}\frac{E_N}{N}\leq E_{nlS}$.\\

In order to prove that $\lim_{N \uparrow \infty}\frac{E_N}{N}\geq E_{nlS}$ let us introduce the functions $u_m^N$ such that
\begin{equation}\label{scattering2}
u_m^N({\bf r})=\frac{m^{3}}{ m-1}u_0^N(m{\bf r})
\end{equation}
and
\begin{equation}\label{UV}
u_0^N({\bf r})={N^{3\beta-3}}v_0(N^{\beta-1}({\bf r}))
\end{equation}
By using the previous equations we obtain
\begin{equation} 
u_N^N({\bf r})=v_N({\bf r})=\frac{N^{3\beta}}{ N-1}v_0(N^{\beta}({\bf r})),\quad 0 \le \beta  < 1
\end{equation}
By relevant results on the convergence of ground-state energies in the GP scaling limit (see \cite{Lieb,LiebYng} which corresponds to considering $v_N$ as in equation \eqref{eq:vn})
if we denote by $E_m^N$ the ground-state energy associated to the potential $u_m^N$ we have that,  for all fixed $N$,
\begin{equation}
\lim_{m\uparrow \infty}\frac{1}{m}{E_m^N}=E^{N}_{GP}
\end{equation}
where $ {E}^{N}_{GP}$ is the minimum energy of 
\begin{equation} \label{GPfuncN}
 {\mathcal E}^N_{GP}[\phi] = \int  (|\nabla \phi({\bf r})|^2 + V(r)|\phi({\bf r})|^2 +  4\pi a_N  |\phi({\bf r})|^4)d{\bf r}
 \end{equation}
where $a_N$ is the scattering length of the potential $u_0^N$ (see \cite{Lieb,LiebBook} for the definition of scattering length).  

First of all we prove that $\lim_{N\uparrow \infty}4\pi a_N=g$.  
We note that by \eqref{UV} and by changing integration variables
\begin{equation} 
\int u_0^N({\bf r})d{\bf r}=\int v_0({\bf r})d{\bf r}=g
\end{equation}
By the upper bound for the scattering length (see \cite{LiebBook} Appendix B) we have
$$g=\int u_0^N({\bf r})d{\bf r}\ge 4\pi a_N$$
We look for an estimate from below for $ 4\pi a_N$ converging to $g$. Let  
us recall how  the scattering length is defined. Denoting by $R_0$ is the maximum radius of the support of $v_0$, let
$$R_N=N^{1-\beta}R_0$$
be the maximum radius of the support of $u_0^N$.
Setting
\begin{equation} \label{GPfuncNR}
{\mathcal E}^N_{R}[\phi] = \int_{B_R}  (|\nabla \phi({\bf r})|^2 + u_0^N({\bf r})|\phi({\bf r})|^2 )d{\bf r}
 \end{equation}
and denoting by
${\tilde  E}^N_{R}$ the minimum energy 
 with
respect to  $\phi$ in $L^2(B_R)$ (with $B_R$ the ball of radius $R$)
subject to the constrain $\|\phi\|_{B_R}=1$, we have for all $R \geq R_N$:

\begin{equation} 
{\tilde  E}^N_{R}=\frac{4\pi a_N R}{(R-a_N)}
\end{equation}

We observe that for $R > N^{1-\beta}R_0$  we can rewrite \eqref{GPfuncNR} as
\begin{equation} \label{GPfuncNRR}
{\mathcal  E}^N_{R}[\phi] = \int_{\frac{B_R}{N^{1-\beta}}}  ({N^{1-\beta}}|\nabla \phi({\bf r})|^2 + v_0({\bf r})|\phi({\bf r})|^2 )d{\bf r}
 \end{equation}
Denoting by
${ E}^N_{R}$ the minimum energy 
 with
respect to  $\phi$ in $L^2(\frac{B_R}{R})$ 
subject to the constrain $\|\phi\|_{\frac{B_R}{R}}=1$, we want to give a lower bound for  ${ E}^N_{R}$.

If $\phi^N_R$ denotes the minimizer of the previous functional we have that

$${ E}^N_{R}\ge  \int_{\frac{B_R}{N}} v_0({\bf r})  |\phi^N_R({\bf r})|^2 d{\bf r}$$

So it is sufficient to give a lower bound for $ |\phi^N_R|$. To this aim let us introduce the potential:
$$ v^k_0({\bf r})=\min(v_0({\bf r}), k), \quad k >0$$

If $E^{N,k}_R$ denotes the minimum energy when the interaction potential is  $v^k_0({\bf r})$, then we note that
\begin{equation}\label{lowerboundR}
 E^{N}_R \ge E^{N,k}_R\ge \int_{\frac{B_R}{N^{1-\beta}}} v^k_0({\bf r})  |\phi^{N,k}_R({\bf r})|^2 d{\bf r}
\end{equation}
where the function $\phi^{N,k}_R$ satisfies the equation $L^k_N(\phi^{N,k}_R)=0$ with $L^k_N  $ given by
$$ L^k_N( \phi)=N^{1-\beta}\triangle \phi +v^k_0 \phi$$
By the maximum principle if 
$$ L^k_N( \phi)\le 0, \quad \quad \phi \ge 0, \quad \|\phi\|_{\frac{B_R}{N^{1-\beta}}}=1$$
then $ \phi \le  \phi^{N,k}_R.$

We choose $\phi_R^{N,k,\epsilon}=C+\epsilon ||x^2||$ where $\epsilon$ is such that:
$$ -3 N^{1-\beta}\epsilon +(C+\epsilon R_0^2)k \le 0$$
When $N\uparrow \infty$ we can choose $\epsilon \downarrow 0$ and $C$ such that
$$ \left(C+\epsilon \frac{R_0^2}{N^{1-\beta}}\right)=1$$
In particular we have:
$$\phi^{N,k}_R \ge \phi_R^{N,k,\epsilon}\ge 1-\epsilon  \frac{R^2}{N^{1-\beta}}$$
Therefore from \eqref{lowerboundR} we get
$$ E^{N}_R \ge E^{N,k}_R\ge \left(\int v_0^k({\bf r})d{\bf r}\right)\left(1-\epsilon  \frac{R^2}{N^{1-\beta}}\right)$$
Sending $\epsilon \downarrow 0$, $N\uparrow \infty$, $k\uparrow \infty$ and $R\uparrow \infty$ we obtain
$$\lim_{N\uparrow \infty}4\pi a_N\ge \lim_{N\uparrow \infty R\uparrow \infty}\left(\frac{4\pi a_NR}{R-a_N}\right)\ge$$
$$\ge \lim_{k\uparrow \infty}\int v_0^k({\bf r})d{\bf r}=\int v_0({\bf r})d{\bf r}.$$
Therefore
$$\lim_{N\uparrow \infty}4\pi a_N=\int v_0({\bf r})d{\bf r}.$$
\\

Let us now consider the following functional
\begin{equation} \label{GPfuncNk}
 \mathcal{E}^{N,R}_{GP}[\phi] = \int (|\nabla \phi({\bf r})|^2 + V(r)|\phi({\bf r})|^2 +  4\pi a_N  |\phi({\bf r})|^4)d{\bf r}
\end{equation}
and let $E^{N,R}_{GP}=\mathcal{E}^{N,R}_{GP}[\phi]$ where $\phi$ is in $H^1({B_R})$ 
subject to the constrain $\|\phi\|=1$ and with Neumann boundary conditions.

In \cite{Lieb} it has be proven that
\begin{equation} \label{LimEN}
{E_m^N}\ge E^{N,R}_{GP}(1-C(R,\phi^{N,R}_{GP})m^{-\frac{1}{10}})
\end{equation}
where $C(R,\phi^{N,R}_{GP})$ is a continuous, locally bounded function of $R$ and of the minimizer $\phi^{N,R}_{GP}$ with respect to the norm 
$\|\phi_{GP}\|_{L^{\infty}}+\|\nabla \phi_{GP}\|_{L^{\infty}} .$
Indeed in the proof of Theorem 4.1 in \cite{Lieb} we can see that $\phi^{N,R}_{GP}$ depends from both
$$ \min_{\Lambda_L}|\phi^{N,R}_{GP}|^2,\quad  \max_{\Lambda_L}|\phi^{N,R}_{GP}|^2- \min_{\Lambda_L}|\phi^{N,R}_{GP}|^2$$
where ${\Lambda_L}$ is the box of length $L$ and that the two previous quantities are continuous when  $\phi^{N,R}_{GP}$ varies in a continuous way with respect to the norm $\|\phi_{GP}\|_{L^{\infty}}+\|\nabla \phi_{GP}\|_{L^{\infty}} .$
Consequently if we are able to prove that $C(R,\phi^{N,R}_{GP})$ is bounded with respect to $N$ then we obtain
\begin{equation} 
\lim_{N\uparrow \infty}\frac{E_N}{N}=\lim_{N\uparrow \infty}\frac{E_N^N}{N}\ge \lim_{N\uparrow \infty} \left[E^{N,R}_{GP}(1-C(R,\phi^{N,R}_{GP})N^{-\frac{1}{10}})\right]=\lim_{N\uparrow \infty}E^{N,R}_{GP}.
\end{equation}
On the other hand, since $4\pi a_N \rightarrow g$, we have that $\lim_{N\uparrow \infty}E^{N,R}_{GP}=E^R_{nlS}$, where $E^R_{nlS}$ is the minimum of the functional
\begin{equation}
 {\mathcal E}^{R}_{nlS}[\phi] = \int  (|\nabla \phi({\bf r})|^2 + V(r)|\phi({\bf r})|^2 +  g  |\phi({\bf r})|^4)d{\bf r}
\end{equation} 
 with
respect to  $\phi$ in $H^1({B_R})$ 
subject to the constraint $\|\phi\|=1$ with Neumann boundary conditions. \\
Since $ E^{R}_{nlS}$ converges to $ E_{nlS}$ when ${N\uparrow \infty}$ we finally have
\begin{equation} 
\lim_{N\uparrow \infty}\frac{E_N}{N}\ge  E_{nlS}
\end{equation}\\

It remains to prove that $C(R,\phi^{N,R}_{GP})$ is bounded with respect to $N$. We provide this by showing that 
\begin{equation} \label{phiconvergence}
\lim_{N\uparrow \infty}\phi^{N,R}_{GP}=\phi^R_{nlS}
\end{equation}
with respect to the norm $\|\cdot\|_{L^{\infty}}+\|\nabla\cdot\|_{L^{\infty}}$.\\
Let us first note that $\phi^{N,R}_{GP}$ satisfies the equation
 \begin{equation}\label{GPNR}
  -\triangle \phi^{N,R}_{GP} +V\phi^{N,R}_{GP} + 8\pi a_N |\phi^{N,R}_{GP}|^2 \phi^{N,R}_{GP} =\mu^R_{GP}(a_N)\phi^{N,R}_{GP}
 \end{equation}
from which we obtain that $\triangle \phi^{N,R}_{GP}\in L^{3/2+\epsilon}(B_R)$ with the bounds
$$\|\triangle \phi^{N,R}_{GP} \|_{L^{3/2+\epsilon}(B_R)} \leq C_1 (a_N^{3/2}+1) \int_{B_R}{(\phi^{N,R}_{GP}({\bf r}))^4d{\bf r}} \leq C_1
\frac{a_N^{3/2}+1}{4\pi a_N} E^{N,R}_{GP},$$
where $C_1$ is positive constant independent  of $N$. Using a bootstrap argument we obtain that $\|\phi^{N,R}_{GP}  \|_{C^{2-\epsilon}} \leq F(E^{N,R}_{GP})$ where $F$ is a continuous increasing function from $\mathbb{R}_+$ into itself and $C^{2-\epsilon}$ is the space of $2-\epsilon$ H\"older functions with $0<\epsilon<1$. Since $a_N\rightarrow g$ and $E^{N,R}_{GP}$ depends continuously on $a_N$, we have that $\sup_N \|\phi^{N,R}_{GP}  \|_{C^{2-\epsilon}} <+\infty$, and so $\phi^{N,R}_{GP}$ stays in a compact set of $C^1$ with respect $\|\cdot\|_{L^{\infty}}+\|\nabla\cdot\|_{L^{\infty}}$ norm. Since equation \eqref{GPNR} has a unique solution and by Berge Maximum Theorem (see \cite{Aliprantis} Theorem 17.31) the map $a_N \longmapsto \phi^{N,R}_{GP}$ is continuous with respect to $\|\cdot\|_{L^{\infty}}+\|\nabla\cdot\|_{L^{\infty}}$ norm, we have that $\lim_{N\uparrow \infty}\phi^{N,R}_{GP}=\phi^R_{nlS}$ in $C^1$ and so $C(R,\phi^{N,R}_{GP})$ is bounded with respect to $N$.
\end{proof}

We aim at characterizing the limit of
the single components of the ground state energy $E_N[\Psi_N]$.\\
Let us introduce the following energy functionals, for any $\lambda>0$:
\begin{equation}\label{energy1}
{\mathcal E}^{1}[\Psi_N,\lambda]=\int \sum_{i=1}^N(|\nabla_i \Psi_N|^2+\lambda V({\bf r}_i))| \Psi_N|^2+
\sum_{1\leq i< j\leq N} v({\bf r}_i-{\bf r}_j)| \Psi_N|^2d{\bf r}_1\cdot\cdot\cdot{\bf
r}_N,
\end{equation}
\begin{equation}\label{energy2}
{\mathcal E}^{2}[\Psi_N,\lambda]=\int \sum_{i=1}^N(|\nabla_i \Psi_N|^2+V({\bf r}_i))| \Psi_N|^2+
\sum_{1\leq i< j\leq N}\lambda v({\bf r}_i-{\bf r}_j)| \Psi_N|^2d{\bf r}_1\cdot\cdot\cdot{\bf
r}_N,
\end{equation}
\begin{equation} \label{GPfunc1}
{\mathcal E}^{1}_{nlS}[\phi,\lambda] = \int  \left(\frac {\hbar^2}{2m}|\nabla \phi({\bf r})|^2 + \lambda V(r)|\phi({\bf r})|^2 +  g  |\phi({\bf r})|^4\right)d{\bf r}, \quad \beta >0,
 \end{equation}
\begin{equation} \label{GPfunc2}
{\mathcal E}^{2}_{nlS}[\phi,\lambda] = \int  \left(\frac {\hbar^2}{2m}|\nabla \phi({\bf r})|^2 + V(r)|\phi({\bf r})|^2 + \lambda g  |\phi({\bf r})|^4\right)d{\bf r}, \quad \beta >0
 \end{equation}
We denote by $E^1_N(\lambda)$, $E^2_N(\lambda)$, $E^1_{nlS}(\lambda)$, $E^2_{nlS}(\lambda)$ the minimum of the four above energy functionals respectively.\\

We introduce the following hypothesis.
\begin{hypothesis}
For $\lambda$ in an neighbour  of $1$
\begin{equation} \label{LimE1}
\lim_{N\rightarrow \infty}\frac{E^1_N(\lambda)}{N}=E^1_{nlS}(\lambda)
\end{equation}
\begin{equation} \label{LimE2}
\lim_{N\rightarrow \infty}\frac{E^2_N(\lambda)}{N}=E^2_{nlS}(\lambda)
\end{equation}
\end{hypothesis}
If Theorem \ref{theorem1} holds, then {Hypothesis A} is true when $v({\bf r})\in L^1(\R^3)$, positive and with compact support. Indeed $g_{\lambda}=\int_{\R^3}\lambda v({\bf r})d{\bf r}=\lambda g$.

\begin{theorem}[Energy Components]\label{theorem2} 
Under the same hypothesis h1),h2) as in Theorem \ref{theorem1}, and the {Hypothesis A}, let us suppose that ${\mathcal E}^{1}_{nlS}[\phi,\lambda]$  and  ${\mathcal E}^{2}_{nlS}[\phi,\lambda]$ admit a unique minimizer in a neighbour of $\lambda=1$.  Then
\begin{multline}\label{E1}
\lim _{N\uparrow \infty} \int_{\R^{3}} \int_{\R^{3N-3}} |\nabla_1
\Psi_N({\bf r}_1,...,{\bf r}_N)|^2d{\bf r}_1
\cdot\cdot\cdot d{\bf
r}_N=\int_{\R^3} |\nabla \phi_{nlS}({\bf r})|^2 d{\bf r}
\end{multline}
and, moreover,
\begin{equation}\label{E2}
\lim _{N\uparrow \infty} \int_{\R^{3}} \int_{\R^{3N-3}}  V({\bf
r}_1)|\Psi_N({\bf r}_1,...,{\bf r}_N)|^2 d{\bf r}_1 \cdot\cdot\cdot
d{\bf r}_N= \int_{\R^{3}} V({\bf r})|\phi_{nlS}({\bf r})|^2 d{\bf r}
\end{equation}
\begin{equation}\label{E3}
\lim _{N\uparrow \infty}{\frac{1}{2}}\sum_{j=2}^{N} \int_{\R^{3}}
\int_{\R^{3N-3}}  v(|{\bf r}_1-{\bf r}_j|)|\Psi_N({\bf
r}_1,...,{\bf r}_N)|^2d{\bf r}_1\cdot\cdot\cdot d{\bf r}_N=\\
g\int_{\R^{3}} |\phi_{nlS}({\bf r})|^4 d{\bf r}
\end{equation}
where $\Psi_N$ and $\phi_{nlS}$ are the unique minimizers of ${\mathcal E}_N$ and ${\mathcal E}_{nlS}$ respectively.
\end{theorem} \vskip 10pt

\begin{proof}[Proof of Theorem \ref{theorem2}]
Since ${\mathcal E}^{1}[\Psi_N,\lambda], {\mathcal E}^{2}[\Psi_N,\lambda], {\mathcal E}^{1}_{nlS}[\phi,\lambda], {\mathcal E}^{2}_{nlS}[\phi,\lambda]$ 
depend linearly by $\lambda$ we have that  $E^1_N(\lambda)$, $E^2_N(\lambda)$, $E^1_{nlS}(\lambda)$, $E^2_{nlS}(\lambda)$ are concave in $\lambda$ (as minimizers of functionals which are linear in 
$\lambda$). 

Thanks to Hypothesis A and using the fact that $E^1_N(\lambda)$, $E^2_N(\lambda)$, $E^1_{nlS}(\lambda)$, $E^2_{nlS}(\lambda)$ are concave we have
\begin{equation} 
\lim_{N\uparrow \infty}\frac{1}{N}{\partial_{\lambda}E^1_N(\lambda)}=\partial_{\lambda}E^1_{nlS}(\lambda)
\end{equation}
\begin{equation} 
\lim_{N\uparrow \infty}\frac{1}{N}{\partial_{\lambda}E^2_N(\lambda)}=\partial_{\lambda}E^2_{nlS}(\lambda)
\end{equation}
whenever $\partial_{\lambda}E^1_{nlS}(\lambda)$ and $\partial_{\lambda}E^2_{nlS}(\lambda)$ are well defined. By using the Hellman-Feynman Principle  we obtain
\begin{equation} 
\frac{\partial_{\lambda}E^1_N(\lambda)}{N}= \int_{\R^{3}} \sum_{i=1}^{N}  V({\bf
r}_1)|\Psi_N^{1,\lambda}({\bf r}_1,...,{\bf r}_N)|^2 d{\bf r}_1 \cdot\cdot\cdot
d{\bf r}_N
\end{equation}
\begin{equation} 
\frac{\partial_{\lambda}E^2_N(\lambda)}{N}= \int_{\R^{3}} \sum_{ i< j}   v(|{\bf r}_i-{\bf r}_j|)|\Psi_N^{2,\lambda}({\bf r}_1,...,{\bf r}_N)|^2 d{\bf r}_1 \cdot\cdot\cdot
d{\bf r}_N
\end{equation}
where $\Psi_N^{i,\lambda}$ is the minimizer of  ${\mathcal E}^{i}[\Psi_N,\lambda]$ with $i=1,2$.\\

For calculating the derivatives of (for example) $E^{2}_{nlS}(\lambda)$ we recall that
 \begin{equation}\label{GP2lambda}
  -\triangle \phi^{2,\lambda}_{nlS} +V\phi^{2,\lambda}_{nlS} + 2\lambda g |\phi^{2,\lambda}_{nlS}|^2 \phi^{2,\lambda}_{nlS} =\mu_{nlS}(\lambda)\phi^{2,\lambda}_{nlS},
 \end{equation}
 where $\mu_{nlS}(\lambda)= E^{2}_{nlS}(\lambda)+\lambda g \int_{\R^3} |\phi^{2,\lambda}_{nlS}|^4d{\bf r}$. It is simple to see that the function $\mu_{nlS}(\lambda)$ is continuous with respect to $\lambda$ and it is differentiable whenever $E^{2}_{nlS}(\lambda)$ is differentiable.

Since ${\mathcal E}^{2}_{nlS}[\lambda]$ has only one minimum, using a reasoning similar to the one of the proof of Theorem \ref{theorem1} and by Berge Maximum Theorem (see \cite{Aliprantis}) the map $\lambda \longmapsto \phi_{nlS}^{2,\lambda}$ is a continuous function in $\lambda$ and also differentiable whenever $\partial_{\lambda}E^2_{nlS}(\lambda)$ is well defined.

By differentiating \eqref{GP2lambda} with respect to $\lambda$, when $E^2_{nlS}(\lambda)$ is differentiable, and by multiplying for $\phi_{nlS}^{2,\lambda}$ and finally by integrating we obtain

 \begin{multline}
  -\int  \phi^{2,\lambda}_{nlS}\triangle \phi^{2,\lambda}_{nlS} +\int V\phi^{2,\lambda}_{nlS} \partial_{\lambda} \phi^{2,\lambda}_{nlS} + 2\lambda g \int |\phi^{2,\lambda}_{nlS}|^3  \partial_{\lambda}\phi^{2,\lambda}_{nlS} =\\
=\mu_{nlS}(\lambda)\int \phi^{2,\lambda}_{nlS} \partial_{\lambda} \phi^{2,\lambda}_{nlS} +  \partial_{\lambda}E^{2}_{nlS}[\lambda]-g\int  |\phi^{2,\lambda}_{nlS}|^4
 \end{multline}
Since equation \eqref{GP2lambda} holds, we deduce that 
\begin{equation}
 \partial_{\lambda}E^{2}_{nlS}(\lambda)=g\int  |\phi^{2,\lambda}_{nlS}|^4
\end{equation}
Now, $\phi^{2,\lambda}_{nlS}$ is continuous with respect to $\lambda$ and since $E^{2}_{nlS}(\lambda)$ is concave it is almost everywhere differentiable. Then $E^{2}_{nlS}(\lambda)$ is continuously differentiable $\forall \lambda$. By the concavity of the previous functions we finally obtain
\begin{equation} 
\lim_{N\uparrow \infty} \int_{\R^{3N}} \sum_{i=1}^{N}  V({\bf
r}_1)|\Psi_N^{1,\lambda}({\bf r}_1,...,{\bf r}_N) |^2d{\bf r}_1 \cdot\cdot\cdot
d{\bf r}_N=\int V({\bf
r}_1)|\phi_{nlS}^{1,\lambda}({\bf r})|^2 d{\bf r}
\end{equation}
\begin{equation} 
\lim_{N\rightarrow \infty} \int_{\R^{3N}} \sum_{ i< j}   v(|{\bf r}_i-{\bf r}_j|)|\Psi_N^{2,\lambda}({\bf r}_1,...,{\bf r}_N)|^2 d{\bf r}_1 \cdot\cdot\cdot
d{\bf r}_N=\int g |\phi_{nlS}^{2,\lambda}({\bf r})|^4d{\bf r} 
\end{equation}
\end{proof}

\section{ Stochastic mechanics and Bose-Einstein condensation}\label{section_nelson}

\vskip5pt

Nelson's Stochastic Mechanics  allows to study quantum phenomena using a
well determined class of diffusion processes
(\cite{Nelson1,Nelson2,Carlen,Carlen2}). See
\cite{CarlenN} for a more recent review on Stochastic Mechanics.

We will briefly introduce the class of \textit{Nelson} diffusions
which are associated to a solution of a Schr\"{o}dinger equation.

Let the complex-valued function (\textit{wave function})
$\psi(x,t)$ be a solution of the equation:
\begin{equation}
i \partial_t \psi(x,t)=H\psi(x,t),\quad t\in {\R}, \quad x\in
{\R}^d,
\end{equation}
\noindent  with $\psi(x,0)=\psi_0(x)$, corresponding to the
Hamiltonian operator:
$$ H=-\frac{\hbar^2}{2m}\triangle +V(x),$$
\noindent where $\hbar$ denotes the reduced Planck constant, $m$
denotes the mass of a particle, and $V$ is some scalar potential such that $H$ is realized as a self-adjoint lower semibounded operator on a dense domain $\mathcal{D}(H) \subset L^2(\mathbb{R}^d)$.

Let us set:
\begin{equation}\label{osmvelocity}
u(x,t):=\mathfrak{Re}\left[\frac{\nabla\psi(x,t)}{\psi(x,t)}\right]
\end{equation}
\begin{equation}\label{curvelocity}
v(x,t):=\mathfrak{Im}\left[ \frac{\nabla\psi(x,t)}{\psi(x,t)}\right]
\end{equation}
\noindent when $\psi(x,t)\neq 0$ and, otherwise, set both $u(x,t)
$ and $v(x,t) $ to be  equal to zero.  Let us put
\begin{equation}
b(x,t):=u(x,t)+v(x,t)
\end{equation}

Let $(\Omega, \mathcal F, \mathcal F_t,X_t) $, $t\geq 0$, with
$\Omega=C(\Real_+,\Real^d)$, be the evaluation stochastic process
$X_t(\omega)=\omega(t)$, with $\mathcal F_t=\sigma(X_s, s\leq t)$
the natural filtration.

Carlen (\cite{Carlen,Carlen2,Carlen3}) proved that
if  $||\nabla \Psi||_2^2<\infty$
then there exists a unique Borel
probability measure $\mathbb P$ on $\Omega$ such that
\begin{enumerate}[i)]

\item $(\Omega, \mathcal F, \mathcal F_t,X_t,\mathbb P) $ is a Markov
process;

\item the image of $\mathbb P$ under $X_t$ has a density
$\rho(t,x):=\rho_t(x)$, for every $t>0$;

\item $W_t:=X_t-X_0-\int_0^tb(X_s,s)ds$  is a $(\mathbb P,\mathcal F_t)$-Brownian Motion.

\end{enumerate}

\vskip 10pt The continuity problem for the above
Nelson-Carlen map (from solutions of Schr\"{o}dinger equations to
probability measures on the path space given by the laws of the corresponding
Nelson-Carlen diffusions $X_t$) is investigated in \cite{DPos}. For a
generalization to the case of Hamiltonian operators with magnetic
potential see \cite{PosU}.

From now on we will mainly consider the case where $d=3$.

We adopt the following notations: capital letters for stochastic
processes or, otherwise, we will explicitly specify them, $\hat
Y=(Y_1,...,Y_N)$ to denote arrays in $\R^{3N}$, $N\in {\mathbb
N}$, and bold letters for vectors in $\R^3$.

\vskip4pt

Here we precisely identify the interacting diffusions system
rigorously associated to the ground state solution $\Psi^0_N$ of
the Hamiltonian \eqref{HN}.

Introducing the probability space $(\Omega^N, \mathcal F^N,
\mathcal F^N_t, \hat{Y}_t) $, with $\hat Y_t(\omega)=\omega(t)$ the
evaluation stochastic process, with $\mathcal F^N_t=\sigma(Y_s,
s\leq t)$ the natural filtration, since $||\nabla \Psi_N||^2 \le \infty$, then by Carlen's Theorem there
exists an unique Borel probability measure ${\mathbb P}_N$ such
that
\begin{enumerate}[ i)]

\item $(\Omega^N, \mathcal F^N, \mathcal F^N_t,\hat Y_t,\mathbb P_N) $
is a Markov process;

\item the image of $\mathbb P_N$ under $\hat Y_t$ has density
$\rho_N({\bf r})$;

\item $\hat W_t:=\hat Y_t-\hat Y_0-\int_0^tb_N(\hat Y_s)ds$
 where
$$b_N(\hat Y_t):=\frac
{\nabla^{(N)}\Psi^0_N}{\Psi^0_N}=\frac{1}{2}\frac
{\nabla^{(N)}\rho_N}{\rho_N}.$$
\end{enumerate}

The stationary probability measure ${\mathbb P}_N$ with density
$\rho_N$ can be alternatively defined as the one of the Markov diffusion process (properly) associated to the
Dirichlet form (\cite{AU,Fukushima,Fukushima2,MR}):

\begin{equation}
\epsilon_{\rho_N}(f,g):=\frac{1}{2}\int_{\Real^{3N}}\nabla f({
r})\cdot\nabla g({ r})\rho_Nd{ r}^{3N}\quad \quad f,g \in
C_c^{\infty}(\Real^{3N})
\end{equation}

\section{A total variation convergence of the one-particle measure on the path space}\label{section_totalvariation}

In the present section we focus on a convergence problem for the
probability measure on the path
space corresponding to the one-particle process.

We consider the measurable space $(\Omega^N,\mathcal F ^N)$  where
$\Omega$ is $C(\R^+ \rightarrow \R^{3})$, $N\in {\mathbb N}$ and
$\mathcal F $ is its Borel sigma-algebra as introduced in Section
3. We denote by $\hat Y := (Y_1,\dots,Y_N)$ the coordinate process
and by $\mathcal F ^N_t$ the natural filtration.

Let us  introduce a process $X^{nlS}$ with invariant density
$\rho_{nlS}$, that is we assume that $X^{nlS}$ is a weak solution of the SDE 
\begin{equation} \label{SDE}
dX^{nlS}_t:= u_{nlS}(X^{nlS}_t)dt + \left(\frac {\hbar}{m}\right)^{\frac 1
2}dW_t
\end{equation}

\noindent where,
$$
u_{nlS} := \frac {1}{2}\frac {\nabla \rho_{nlS}}{\rho_{nlS}}
$$

The vector field $u_{nlS}$ is well defined since $\rho_{nlS}$ is continuously differentiable and strictly positive under hypothesis h1) and h2) (see \cite{Lieb,LiebBook}).  We denote again by ${\mathbb P}_N$  the measure corresponding to the weak solution of the
$3N$- dimensional stochastic differential equation

\begin{equation} \label {SDEa}
\hat Y_t-\hat Y_0 = \int_0^t \hat b ^N (\hat Y_s)ds +\hat W_t
\end{equation}

\noindent where $ \hat Y_0$  is a random variable with probability
density equal to $ \rho_N$, while  $\hat W_t$ is a
$3N$-dimensional $\mathbb P_N$ standard
Brownian motion.

In this section we use the shorthand notation $ \hat {b}^N_s=:\hat
{b}^N(\hat {Y}_s) $.

We denote by ${\mathbb
P}^N_{nlS}$ the measure corresponding to the weak solution of the
$3N$- dimensional stochastic differential equation

\begin{equation} \label {SDEb}
\hat Y_t-\hat Y_0 = \int_0^t \hat u_{nlS} (\hat Y_s)ds +\hat W_t',
\end{equation}

\noindent where
$$
\hat u_{nlS}({\bf r}_1,\cdots,{\bf r}_N)=(u_{nlS}({\bf
r}_1),\cdots,u_{nlS}({\bf r}_N)),
$$
\noindent $ \hat Y_0$  is a random variable with probability
density equal to $ \rho_N$ and  $\hat W_t'$ is a
$3N$-dimensional  $\mathbb P^N_{nlS}$ standard
Brownian motion.

Following \cite{MU}, the next lemma  computes the one-particle relative entropy between the
three-dimensional {\it one-particle} non markovian diffusions $Y_1$
and $X^{nlS}$.

\begin{lemma}\label{lemma_entropy}
Under hypothesis h1) and h2) we have
\begin{multline}\label{equation_drift} 
{\mathbb E}_{\mathbb P_N}[  \| b ^N_1(\hat Y_s) -  u^{nlS}(Y_1(s))\|^2 ]=\int{\frac{|\nabla\Psi_N|^2}{2}\text{d}\bold{r}_1...\text{d}\bold{r}_N}
-\mu_{nlS}+\\
+\int{V(\bold{r}_1)|\Psi_N|^2\text{d}\bold{r}_1...\text{d}\bold{r}_N}+2g\int{\phi_{nlS}^2(\bold{r}_1)|\Psi_N|^2\text{d}\bold{r}_1...\text{d}\bold{r}_N}.
\end{multline}
\end{lemma}
\begin{proof}
By simple computation and recalling that $\phi_{nlS}$ is strictly positive and $C^2$ and hypothesis h1) we have
$${\mathbb E}_{\mathbb P_N}[  \| b ^N_1(\hat Y_s) -  u^{nlS}(Y_1(s))\|^2 ]=\frac{1}{2}\int_{\mathbb R^n}{\left|\nabla_1\left( \frac{\Psi_N}{\phi_{nlS}} \right) \right|^2 \phi_{nlS}^2 d{\bf r_1}...d{\bf r_N}}$$
where $\Psi_N$ is the ground state of the $N$ body Hamiltonian \eqref{HN}. \\
We now want to prove that $\int_{\mathbb R^n}{\left|\nabla_1\left( \frac{\Psi_N}{\phi_{nlS}} \right) \right|^2 \phi_{nlS}^2 d{\bf r_1}...d{\bf r_N}}$ is finite and equal to the right hand side of equation \eqref{equation_drift}. In order to prove this, let $\Psi_{R,N}$ be the ground state of Hamiltonian \eqref{HN} restricted on the ball $B_R$, with radius $R$ and centre in $0$, with Dirichlet boundary condition. Using integration by parts, the fact that $\Psi_{N,R}|_{\partial B_R}=0$ and the nlS equation \eqref{GP} we obtain 
\begin{multline}
\frac{1}{2}\int_{B_R}{\left|\nabla_1\left( \frac{\Psi_{N,R}}{\phi_{nlS}} \right) \right|^2 \phi_{nlS}^2 d{\bf r_1}...d{\bf r_N}}\\
= \int_{B_R}{\left(\frac{\left|\nabla_1\Psi_{N,R}\right|^2}{2}-\nabla_1\left(\frac{|\Psi_{N,R}|^2}{\phi_{nlS}}\right)\cdot \nabla_1\phi_{nlS} \right)d{\bf r_1}...d{\bf r_N}}\\
=\int_{B_R}{\frac{|\nabla\Psi_{N,R}|^2}{2}\text{d}\bold{r}_1...\text{d}\bold{r}_N}+\int_{B_R}{V(\bold{r}_1)|\Psi_{N,R}|^2\text{d}\bold{r}_1...\text{d}\bold{r}_N}
+\\
-\mu_{nlS}+2g\int{\phi_{nlS}^2(\bold{r}_1)|\Psi_{N,R}|^2\text{d}\bold{r}_1...\text{d}\bold{r}_N}.
\end{multline}
Using the fact that $|\Psi_{N,R}|^2 \rightarrow |\Psi_N|^2$ weakly and that $\int_{B_R}{\frac{|\nabla\Psi_{N,R}|^2}{2}\text{d}\bold{r}_1...\text{d}\bold{r}_N} \rightarrow \int_{\mathbb{R}^{3N}}{\frac{|\nabla\Psi_{N}|^2}{2}\text{d}\bold{r}_1...\text{d}\bold{r}_N}$ as $R\rightarrow +\infty$ the lemma is proved.
\end{proof}

\begin{remark}\label{remark_entropy}
Lemma \ref{lemma_entropy} has two important consequences. First of all, since  $\Psi^0_N$ is the minimizer of
$E^N[\Psi]$, the following finite energy conditions hold:
\begin{equation}\label {eni}
{\mathbb E}_{\mathbb P_N} \int_0^t \|\hat b ^N_s \|  ^2 ds  < \infty
\end{equation}
\begin{equation}\label{enii}
{\mathbb E}_{\mathbb P_N} \int_0^t\| \ \hat u^{nlS}_s  \| ^2 ds   < \infty .
\end{equation}
Furthermore, by Theorem \ref{theorem2}, we obtain 
$${\mathbb E}_{\mathbb P_N}[  \| b ^N_1(\hat Y_s) -  u^{nlS}(Y_1(s))\|^2 ]\rightarrow 0$$
as $N \rightarrow \infty$.
\end{remark}

\begin{lemma}\label{lemma2}

 The one-particle (or normalized)  relative entropy is given by

 \begin{equation}
 \bar{\mathcal H} (\mathbb P_N, \mathbb P^N_{nlS})|_{\mathcal F_t}
 =\frac 1 2 {\mathbb E}_{\mathbb P_N}\int_0^{t}  \| b ^N_1(\hat Y_s) -  u^{nlS}(Y_1(s))\|^2 d s
 \end{equation}

\end{lemma}

\begin{proof}
The proof is similar to the one proposed for GP limit in \cite{MU}. For this reason we report here only a sketch.

 \noindent  The inequalities \eqref{eni} and \eqref{enii} are  \textit{finite entropy conditions} (see, e.g.
\cite{Follmer}) which imply that $\forall t > 0$

$$\mathbb P_N|_{\mathcal F_t} \ll \hat W|_{\mathcal F_t}, \quad \mathbb P^N_{nlS}|_{\mathcal F_t}\ll \hat W'|_{\mathcal F_t}  $$
(where $\ll$ stands for absolute continuity)
 By Girsanov's theorem, we have, for all $t>0$,
\begin{equation}\label{derivative}
\left. \frac {d\mathbb P_N}{d\mathbb P^N_{nlS}}\right|_{\mathcal F_t}=
  \exp \left\{-\int_0^{t}  (\hat b ^N _s-
 \hat u^{nlS}_s)\cdot
 d \hat W_s+\frac{1}{2} \int_0^ {t}  \|\hat b ^N_s - \hat u^{nlS}_s\|^2ds
\right\},
 \end{equation}

\noindent  where $|.|$ denotes the Euclidean norm in $\R^{3N}$.
The relative entropy restricted to $\mathcal F_t$ reads

 \begin{multline} \label{eq: Entropy}
{\mathcal H}(\mathbb P_N,\mathbb P^N_{nlS}) |_{\mathcal F_{t}}
 =: {\mathbb E}_{\mathbb P_N}\left[
 \log\left. \frac {d\mathbb P_N}{d\mathbb P^N_{nlS}}\right|_{\mathcal F_t}\right]=
 \frac {1}{2}{\mathbb E}_{\mathbb P_N}\int_0^{t}  \|\hat b ^N_s - \hat u^{nlS}_s\|^2 ds
 \end{multline}

 Since under $\mathbb P_N$ the $3N$-dimensional process $\hat Y$ is a solution of \eqref {SDEa}
 with invariant probability density $ \rho_N$ , we can write, recalling also \eqref{eni} and \eqref{enii},
and by using  the symmetry of $\hat b^N $ and $\rho_N$

 \begin{multline}
  \mathcal H(\mathbb P_N,\mathbb P^N_{nlS}) |_{\mathcal F_{t}} =\\
 = \frac 1 2 t
  \int _{\R^{3N}} \sum _{i=1}^{N}\| b ^N_i({\bf r}_1,\dots,{\bf r}_N) -  u_{nlS} ({\bf r}_i)\|^2 \rho_N d {\bf r}_1 \dots d {\bf r}_N =\\
  =
 \frac 1 2 N t\int _{\R^{3N}} \| b ^N_1({\bf r}_1,\dots,{\bf r}_N) -  u_{nlS} ({\bf r}_1)\|^2 \rho_N d {\bf r}_1 \dots d {\bf r}_N =\\
 =\frac 1 2 Nt {\mathbb E}_{\mathbb P_N}\int_0^{t}  \| b ^N_1 (\hat Y_s) -  u_{nlS}(Y_1(s))\|^2 d
s,
 \end{multline}

 By defining the one-particle relative entropy as the {\it normalized relative entropy} we obtain

 \begin{multline}\label{oneparticlerelativentropy}
 \bar{\mathcal H} (\mathbb P_N, \mathbb P^N_{nlS})|_{\mathcal F_t}=:
 \frac 1 N  \mathcal H(\mathbb P_N,\mathbb P^N_{nlS}) |_{\mathcal F_{t}} = \\
 =\frac 1 2 {\mathbb E}_{\mathbb P_N}\int_0^{t}  \| b ^N_1(\hat Y_s) -  u^{nlS}(Y_1(s))\|^2 d s
 \end{multline}

\end{proof}

By  Theorem \ref{theorem2} we deduce that for any $ t
> 0$ the one particle relative entropy converges to zero in the scaling
limit.

\begin{remark}
In the more complicated GP scaling limit the one-particle relative entropy does not converge to zero but to a finite constant. This fact has important consequences for the convergence of the one-particle probability measure. See \cite{DeVU} for the proof of the existence of the limit probability measure, \cite{ADU} for the proof of weak convergence of the one-particle process and \cite{MU1} for the localization phenomenon of the relative entropy.
\end{remark}

The proof  the theorem takes advantage of the following lemma, which represents a useful chain-rule of the relative entropy when the reference measure is a {\it product one}.

\begin{lemma}\label{lemma3} We consider $M=X\times Y$, where $X$ and $Y$ are
Polish spaces. Let ${\mathbb P}$ be a measure on $M$ and ${\mathbb
Q}_1$ and ${\mathbb Q}_2$ probability measures on $X$ and $Y$ respectively. We
denote by $\mathbb Q={\mathbb Q}_1\otimes {\mathbb Q}_2$ the
product measure on $M$ of the measures ${\mathbb Q}_1$ and
${\mathbb Q}_2$ and we suppose that ${\mathbb P} \ll {\mathbb Q}$.
Then we have
\begin{equation}
{\mathcal H}({\mathbb P}|{\mathbb Q})\geq {\mathcal H}({\mathbb
P}_1|{\mathbb Q}_1) + {\mathcal H}({\mathbb P}_2|{\mathbb Q}_2),
\end{equation}
\noindent where ${\mathbb P}_1$ and ${\mathbb P}_2$ are the
marginal probabilities of ${\mathbb P}$.
\end{lemma}
\begin{proof}
The proof can be found in Lemma 5.1 of \cite{DeVU}. 
\end{proof}

\begin{theorem}[Total variation convergence on the path space]\label{theorem3}Under the same hypothesis
h1),h2) as in Theorem \ref{theorem1}, the one-particle measure  ${\mathbb{P}}_N^1$ converges in total variation to
$\mathbb{P}_{nlS}$, the latter being uniquely associate to the non linear Schroedinger functional.
\end{theorem}
\begin{proof}
By Lemma \ref{lemma2} the one-particle relative entropy reads
 \begin{equation}
 \bar{\mathcal H} (\mathbb P_N, \mathbb P^N_{nlS})|_{\mathcal F_t}
 =\frac 1 2 {\mathbb E}_{\mathbb P_N}\int_0^{t}  \| b ^N_1(\hat Y_s) -  u^{nlS}(Y_1(s))\|^2 d s
 \end{equation}
and by Remark \ref{remark_entropy} we obtain 
\begin{equation}\label{equation_limit}
\lim_{N\uparrow \infty}\bar{\mathcal H} (\mathbb P_N, \mathbb P^N_{nlS})|_{\mathcal F_t}=\frac t 2\lim_{N\uparrow \infty}  {\mathbb E}_{\mathbb P_N} [\| b ^N_1(\hat Y_s) -  u^{nlS}(Y_1(s))\|^2]=0
\end{equation}

Let us introduce the total variation distance between the one-particle measure ${\mathbb{P}}_N^1$ and $\mathbb{P}_{nlS}$:
\begin{equation}
d_{TV}({\mathbb P}^1_N,{\mathbb{P}}_{nlS}))|_{\mathcal F_{t}}\:=\sup_{A\in
{\mathcal F}_{t}} |{{\mathbb{P}}}^1_N(A)-{\mathbb{P}}_{nlS}(A)| \\
=\sup_{A\in {\mathcal F}_t}\left|\int_A
\left(\frac{d{\mathbb{P}}^1_N}{d\mathbb{P}_{nlS}}-1\right){d\mathbb{P}_{nlS}}\right|
\end{equation}

By the well-known Csiszar-Kullback inequality
(\cite{Csiszar},\cite{Kullback}), which is valid in arbitrary Polish spaces,

\begin{equation}\label{CKinequality}
d_{TV}({\mathbb P}^1_N,{\mathbb{P}}_{nlS})|_{\mathcal F_{t}}\le \sqrt{2 \mathcal H(\mathbb P^1_N,\mathbb P_{nlS}) |_{\mathcal F_{t}}}.
\end{equation}

By using Lemma \ref{lemma3} we can write

\begin{equation}
{\mathcal H}({\mathbb P^N}|{\mathbb P^N_{nlS}})\geq {\mathcal H}({\mathbb
P}^1_N|{\mathbb P}^1_{nlS}) + {\mathcal H}({\mathbb P}^{N-1}_N|{\mathbb P}^{N-1}_{nlS}),
\end{equation}
and by repeating this procedure we obtain

\begin{equation}\label{onerelativeentropy}
\mathcal H(\mathbb P^1_N,\mathbb P_{nlS}) |_{\mathcal F_{t}}\le \bar{\mathcal H} (\mathbb P_N, \mathbb P^N_{nlS})|_{\mathcal F_t}.
\end{equation}
From equation \eqref{equation_limit} and inequalities \eqref{onerelativeentropy}, \eqref{CKinequality} we deduce that the
sequence ${\mathbb{P}}_N^1$ converges in total variation to
$\mathbb{P}_{nlS}$.

\end{proof}

\section{Strong Kac's chaos on path space}\label{section_kac_chaos}

It is well known that the (usual)  Kac's chaos is implied by statements of the kind given in the following theorem:

\begin{theorem}[Kac's chaos on path space]\label{theorem4}
$\forall T >0$ fixed,$ \forall i=1,2,...,N$ we have
\begin{equation}
\lim_{N\uparrow +\infty}{E}_{{\mathbb P}_N}[\sup_{t\le T}|Y^1_t-Y_t|^2]=0
\end{equation}
\end{theorem}

\begin{proof} Let $T$ be a fixed time.
For all $t\le T$, by exploiting \eqref{SDEa} and \eqref{SDEb}, with $Y^1_t$ and $Y_t$ the first components of $\hat Y_t$ in \eqref{SDEa} and \eqref{SDEb} respectively, we have
\begin{equation}
\sup_{t\le T}|Y^1_t-Y_t|^2\le \left|\int_0^T(b_1^N(\hat{Y}_s)-u_{nlS}(Y^1_s))ds\right|^2
\end{equation}
and by taking the expectation with respect to the measure $\mathbb{P}_N$ we obtain
\begin{multline}
\mathbb{E}_{\mathbb P_N}[\sup_{t\le T}|Y^1_t-Y_t|^2]\le \mathbb{E}_{\mathbb P_N}\left[ \left|\int_0^T(b_1^N(\hat{Y}_s)-u_{nlS}(Y^1_s))ds\right|^2\right]=\\
=\mathbb{E}_{\mathbb P_N}\left[T^2\left(\frac{1}{T} \int_0^T(b_1^N(\hat{Y}_s)-u_{nlS}(Y^1_s))ds\right)^2\right]\le\\ 
\le\mathbb{E}_{\mathbb P_N}\left[T^2 \frac{1}{T} \int_0^T(b_1^N(\hat{Y}_s)-u_{nlS}(Y^1_s))^2ds\right]=\\=T^2 \mathbb{E}_{\mathbb P_N}[(b_1^N(\hat{Y}_s)-u_{nlS}(Y^1_s))^2)] 
\end{multline}
where we have applied Jensen's inequality and we have taken into account the stationarity of both our systems.
By Remark \ref{remark_entropy} we deduce the statement of the theorem.
\end{proof}

\begin{remark}
The Kac's chaos on path space implies the Kac's chaos for the fixed time marginal probability densities, that is for any fixed $k \in {\mathbb N}$, the measures $\rho^k_N({\bf r}_1,\dots,{\bf r}_N)d{\bf r}_1\dots d{\bf r}_N$ converges weakly to $\rho^{\otimes k}({\bf r}) d{\bf r}$.  The latter can be derived by the complete Bose-Einstein Condensation (see \cite{Ugolini}).
\end{remark} 
Let us denote by $ {{\mathbb P}^k_N}$ the measure $ {{\mathbb P}_N}$ restricted with respect to the $\sigma$-algebra generated by (any)  $k$ particles (due to the symmetry).

We prove a result which states that when the normalized relative entropy goes to zero, also the relative entropy between any $k$ marginal probability measures goes to zero.

\begin{lemma}\label{lemma4}
If for all fixed $t >0$, when ${N\uparrow +\infty}$
 \begin{equation}\label{normrelativeentropy}
 \bar{\mathcal H} (\mathbb P_N, \mathbb P^N_{nlS})|_{\mathcal F_t} \rightarrow  0
 \end{equation}
 then, for every fixed $k\in {\mathbb N}$,
\begin{equation}
\lim_{N\uparrow +\infty}{\mathcal H} (\mathbb P^k_N, \mathbb P^k_{nlS}))|_{\mathcal F_t}=0
\end{equation}

\end{lemma}

\begin{proof}
We prove the statement by induction on $k$. By simplicity of notations we do not write the restriction to ${\mathcal F_t}$.  The assertion is true for $k=1$ as we have shown  in \eqref{onerelativeentropy}.

Let us write $N=kN_k+r_k$ and suppose that the statement is true for any $r_k <
 k$. Applying  Lemma \ref{lemma3} we have 
\begin{equation}
{\mathcal H}({\mathbb P_N}|{\mathbb P^N_{nlS}})\geq N_k {\mathcal H}({\mathbb
P}^k_N|{\mathbb P}^k_{nlS}) + {\mathcal H}({\mathbb P}^{r_k}_N|{\mathbb P}^{r_k}_{nlS}),
\end{equation}
which implies:
\begin{multline}
{\mathcal H}({\mathbb
P}^k_N|{\mathbb P}^k_{nlS}) \leq \frac{(k+1)}{N_k} \left\{{\mathcal H}({\mathbb P_N}|{\mathbb P^N_{nlS}})+
\sum^{k-1}_{r=1} {\mathcal H}({\mathbb P}^{r}_N|{\mathbb P}^{r}_{nlS})\right\}\\
= (k+1)\frac{N}{N_k}\bar {\mathcal H}({\mathbb P_N}|{\mathbb P^N_{nlS}})+
\frac{(k+1)}{N_k}\sum^{k-1}_{r=1} {\mathcal H}({\mathbb P}^{r}_N|{\mathbb P}^{r}_{nlS})
\end{multline}
Since $N\uparrow +\infty$ if and only if  $N_k\uparrow +\infty$, by induction hypothesis we obtain the result.

\end{proof}

\begin{theorem}[A strong form of  Kac's chaos on path space]\label{theorem5}
Under the same hypothesis $h_1), h_2)$  as in Theorem \ref{theorem1},
for all fixed $t >0$,$ \forall k\in {\mathbb N}$ we have
\begin{equation}
\lim_{N\uparrow +\infty}d_{TV}({\mathbb P}^k_N,{\mathbb{P}}^k_{nlS})|_{\mathcal F_{t}}=0
\end{equation}
\end{theorem}

\begin{proof}
 Since by Lemma \ref{lemma2}
 \begin{equation}
 \bar{\mathcal H} (\mathbb P_N, \mathbb P^N_{nlS})|_{\mathcal F_t}
 =\frac 1 2 {\mathbb E}_{\mathbb P_N}\int_0^{t}  \| b ^N_1(\hat Y_s) -  u^{nlS}(Y_1(s))\|^2 d s
 \end{equation}
by Remark \ref{remark_entropy} we obtain that
\begin{equation}
\lim_{N\uparrow +\infty}\bar {\mathcal H} (\mathbb P^k_N, \mathbb P^k_{nlS})=0
\end{equation}
 By Lemma \ref{lemma4} it follows that for every fixed $k$,
\begin{equation}
\lim_{N\uparrow +\infty}{\mathcal H} (\mathbb P^k_N, \mathbb P^k_{nlS})=0
\end{equation}
Since by the Csiszar-Kullback inequality we have:
\begin{equation}\label{CKinequality}
d_{TV}({\mathbb P}^k_N,{\mathbb{P}}^k_{nlS})|_{\mathcal F_{t}}\leq \sqrt{2 \mathcal H(\mathbb P^k_N,\mathbb P^k_{nlS}) |_{\mathcal F_{t}}},
\end{equation}
we obtain the result.
\end{proof}

\section{Fisher information chaos in Bose-Einstein Condensation}\label{section_fisher}

In this section
we consider the symmetric
probability law $G^N$ of our $N$ interacting diffusions on the
product space $\mathbb{R}^{3N}$, which are absolutely continuous with density $\rho_N:=|\Psi^0_N|^2$.
 
First we define the Fisher information associated to a probability measure $G$ on a space $\mathbb{R}^{3n}$.

\begin{definition}
For $G\in W^{1,1}(S^n)$ we put
$$I_n(G):=\int_{\mathbb{R}^{3n}}\frac{|\nabla G|^2}{G}$$
otherwise equal to $+\infty$.
We consider the normalized Fisher information $I:=\frac{1}{n}I_n$
\end{definition}
The Fisher information has the following important properties:

\begin{enumerate}

\item $I$ is proper, convex, l.s.c. (in the sense of the weak convergence of measures) on ${\mathcal P}(\mathbb{R}^{3n}) $ (the space of probability measures on $S^n$).

\item \begin{enumerate} 

\item For $1\le l \le n, \quad  I_l(G_l) \le I(G)$ where $G\in {\mathcal  P}(\mathbb{R}^{3n}) $

\item The (non normalized) Fisher information is super-additive, i.e.
$$ I_n(G)\ge I_l(G_l)+ I_{n-l}(G_{n-l})$$
with (in the case $I_l(G_l)+ I_{n-l}(G_{n-l})<+\infty$) equality if and only if $G=G_lG_{n-l}$

\item If $I(G_1) < +\infty$, the equality $I(G_1)=I(G)$ holds if and only if $G=(G_1)^{\otimes j}$

\end{enumerate}\end{enumerate}

\begin{proof}
For 1) see, e.g., \cite{HaMi}, Lemma 3.5. For 2) see \cite{Cadd} or \cite{HaMi}, Lemma 3.7.
\end{proof}

\begin{proposition}[Fisher's chaos in mean-field BEC]
There is Fisher's chaos in the generic mean-field Bose-Einstein Condensation.
\end{proposition}

\begin{proof}
The Fisher information associated to the symmetric law ${\mathbb P_N}$ is:
$$
 I({G_N})=\frac{1}{N}\int_{\mathbb{R}^{3N}}\left|\frac{\nabla \rho_N}{\rho_N}\right|^2\rho_N d {\bf r}_1 \dots d {\bf r}_N
                             =\frac{4}{N}\int_{\mathbb{R}^{3N}}\left|\frac{\nabla \rho_N}{2\rho_N}\right|^2\rho_N d {\bf r}_1 \dots d {\bf r}_N$$
                             $$=\frac{1}{N}\int_{\mathbb{R}^{3N}}|{\nabla \Psi_N}|^2d {\bf r}_1 \dots d {\bf r}_N$$
By Theorem \ref{theorem2}, in particular by \eqref{E1}, we get 
$$\lim_{N\uparrow \infty} I({G_N})=I({G_{nlS}})$$
Moreover, by  Theorem \ref{theorem2}, we also have that 
$\rho^{(1)}_N$ converges to $\rho_{nlS}$ weakly in $\mathcal P(S)$ as $N\rightarrow +\infty$ and so we have that Fisher's chaos holds. 
\end{proof}
\begin{proposition}[Entropy chaos in mean-field BEC]
Entropy chaos holds in the generic mean-field Bose-Einstein Condensation.
\end{proposition}

\begin{proof}
The proof of the fact that the Fisher's information chaos implies entropy chaos and Kac's chaos can be found in \cite{HaMi}.
\end{proof}

\section*{Acknowledgements}

The first and second authors would like to thank the Department of Mathematics, Universit\`a degli Studi di Milano for the warm hospitality. The second author is supported by the German Research Foundation (DFG) via SFB 1060.


\end{document}